\documentclass[12pt]{amsart}
\usepackage[all]{xy}
\usepackage{amsfonts}
\usepackage{amssymb}
\usepackage{pb-diagram}
\usepackage{enumerate}

\usepackage{color}

\newtheorem{teo}{Theorem}[section]
\newtheorem{lem}[teo]{Lemma}

\theoremstyle{definition}
\newtheorem{dfn}[teo]{Definition}
\newtheorem{rk}[teo]{Remark}
\newtheorem{ex}[teo]{Example}
\def\<{\langle}
\def\>{\rangle}
\def\ss{\subset}

\def\a{\alpha}
\def\b{\beta}

\def\e{\varepsilon}
\def\l{{\lambda}}

\def\F{{\Phi}}

\def\Mm{{\mathbb M}}
\def\C{{\mathbb C}}

\def\Ker{\mathop{\rm Ker}\nolimits}

\def\supp{\operatorname{supp}}

\def\1{\mathbf 1}

\newcommand{\norm}[1]{\left\| #1 \right\|}
\newcommand{\ov}[1]{\overline{#1}}
\newcommand{\til}[1]{\widetilde{#1}}
\newcommand{\wh}[1]{\widehat{#1}}

\def\N{{\mathbb N}}

\begin{document}

\title[Reflexivity criterion for Hilbert $C^*$-modules]
{A reflexivity criterion for Hilbert $C^*$-modules over
commutative $C^*$-algebras}

\author{Michael Frank}
\thanks{This work is a part of the joint DFG-RFBR project (RFBR grant
07-01-91555 / DFG project ''K-Theory, $C^*$-Algebras, and Index
Theory''.)}
\address{HTWK Leipzig, FB IMN, Postfach 301166, D-04251 Leipzig, Germany}
\email{mfrank@imn.htwk-leipzig.de}
\urladdr{
http://www.imn.htwk-leipzig.de/\~{}mfrank}

\author{Vladimir Manuilov}
\thanks{The second named author was also partially supported by the grant HIII\hspace{-2.3ex}%
\rule{1.9ex}{0.07ex}\,-1562.2008.1}
\address{Dept. of Mech. and Math., Moscow State University,
119991 GSP-1  Moscow, Russia\\
{\it and}  Harbin Institute of Technology, Harbin, P. R. China}
\email{manuilov@mech.math.msu.su} \urladdr{
http://mech.math.msu.su/\~{}manuilov}

\author{Evgenij Troitsky}
\thanks{The third named author was also partially supported by the RFBR grant 07-01-00046.}
\address{Dept. of Mech. and Math., Moscow State University,
119991 GSP-1  Moscow, Russia}
\email{troitsky@mech.math.msu.su}
\urladdr{
http://mech.math.msu.su/\~{}troitsky}



\begin{abstract}
A $C^*$-algebra $A$ is $C^*$-reflexive if any countably generated
Hilbert $C^*$-module $M$ over $A$ is $C^*$-reflexive, i.e. the
second dual module $M''$ coincides with $M$. We show that a
commutative $C^*$-algebra $A$ is $C^*$-reflexive if and only if
for any sequence $I_k$ of disjoint non-zero $C^*$-subalgebras, the
canonical inclusion $\oplus_k I_k\subset A$ doesn't extend to an
inclusion of $\prod_k I_k$.

\end{abstract}

\maketitle


\section{Introduction}

The aim of the present paper is to study the $C^*$-reflexivity
property for Hilbert $C^*$-modules over $C^*$-algebras. The
motivation comes from three sources. First, this property appears
in our study of dynamical systems and group actions, where it was
shown that some information about orbits can be detected from
$C^*$-reflexivity of the corresponding Hilbert $C^*$-modules
\cite{FMTZAA,FMT2}. Second, $C^*$-reflexive Hilbert $C^*$-modules
are a natural setting for $A$-bilinear functions on them. Third,
there was a series of papers providing various sufficient
\cite{Mishchenko-Trudy_MIAN,TrofimovUMN1987,FMT2} and necessary
\cite{Paschke2} conditions for Hilbert $C^*$-modules over
commutative $C^*$-algebras to be $C^*$-reflexive. The main result
of this paper is a criterion for $C^*$-reflexivity in the
commutative case.

Let us recall some basic facts about the dual and the second dual
of a Hilbert $C^*$-module \cite{Paschke2} (see also
\cite{MaTroBook}). For a Hilbert $C^*$-module $M$ over a
$C^*$-algebra $A$, the {\em dual} Banach module $M'$ is defined
\cite{Paschke2} as the set of all $A$-module bounded linear maps
from $M$ to $A$ (such maps are called {\em functionals}).
Iterating this procedure, one gets the second dual module $M''$.

There are isometric inclusions $M\ss M''\ss M'$ for any Hilbert
$C^*$-module $M$. The identifications are defined as follows.
First of all we have the map $M\to M'$, $m \mapsto \wh m$, $\wh
m(s)=\<s,m\>$ for any $s\in M$. Then we can define the map $M\to
M''$, $m\mapsto \dot m$, $\dot m (f)=f(m)$ for any $f\in M'$.
Finally, $M''\to M'$, $F\mapsto \til F$ is defined by $\til
F(m)=F(\wh m)$. The $A$-valued inner product of $M$ can be
extended to $M''$ by the formula $\<F,G\>=G(\til F)$ and thus
$M''$ becomes a Hilbert $C^*$-module \cite{Paschke2}.

A Hilbert $C^*$-module $M$ is {\em self-dual} if $M'=M$. There are
very few $C^*$-algebras, for which all Hilbert $C^*$-modules are
self-dual, only finitedimensional $C^*$-algebras have this
property \cite{FrankZAA1990}. {\em $C^*$-reflexivity} (i.e.
$M''=M$) is a more common property. For example, all countably
generated Hilbert $C^*$-modules over the $C^*$-algebra of compact
operators with adjoined unit are $C^*$-reflexive \cite{TrofimovK}.

Due to the Kasparov's stabilization theorem \cite{KaspJO}, any
countably generated Hilbert $C^*$-module over a $C^*$-algebra $A$
is $C^*$-reflexive if the standard Hilbert $A$-module $H_A=l_2(A)$
is $C^*$-reflexive. We call a $C^*$-algebra $A$ {\em
$C^*$-reflexive} if $H_A$ is $C^*$-reflexive.

It was shown by Paschke \cite{Paschke} that infinitedimensional
von Neumann algebras are not $C^*$-reflexive. On the positive, it
is known that $C(X)$ is $C^*$-reflexive for nice spaces $X$.

\begin{teo}\label{mishchenko-trofimov}
Let $X$ be a compact metric space. Then $C(X)$ is $C^*$-reflexive.
\end{teo}

The first version of a proof was given by Mishchenko
\cite{Mishchenko-Trudy_MIAN}. Then Trofimov \cite{TrofimovUMN1987}
realized that the formulation in \cite{Mishchenko-Trudy_MIAN} was
too general and provided a proof for any compact $X$ with a
certain property L, which, in fact, is the same as the property of
being a Baire space. Although the main part of the proof in
\cite{TrofimovUMN1987} is correct, it was overlooked that
implicitly $X$ was assumed to be a {\it metric} space. Trofimov's
proof was corrected in \cite{FMT2}.

Many examples of $C^*$-reflexive modules arising from group
actions were obtained in our previous papers \cite{FMTZAA,FMT2}.

The main result of this paper is the criterion for
$C^*$-reflexivity for commutative $C^*$-algebras, which is given
in either topological or algebraic terms.

\section{Topology preliminaries: the Baire property and
the Stone-\v Cech compactification}\label{sec:topolBaireStone}
\begin{dfn}\label{dfn:Bairespace}(\cite[p.~155]{munkres2000})
A space $X$ is said to be a \emph{Baire space} if the following
condition holds: Given any countable collection $\{A_n\}$ of
closed subsets of $X$ each of which has empty interior in $X$,
their union $\cup A_n$ also has empty interior in $X$.
\end{dfn}

\begin{teo}[Baire category theorem]\label{teo:Bairecateg}
{\rm\cite[Theorem~48.2]{munkres2000}} If $X$ is a compact
Hausdorff space or a complete metric space, then $X$ is a Baire
space.
\end{teo}

\begin{teo}\label{teo:Bairefunc}
{\rm\cite[Theorem~48.5]{munkres2000}} Let $X$ be a space; let $(Y,
d)$ be a metric space. Let $f_n : X \to Y$ be a sequence of
continuous functions such that $f_n(x) \to f(x)$ for all $x \in
X$, where $f : X \to Y$. If $X$ is a Baire space, the set of
points at which $f$ is continuous is dense in $X$.
\end{teo}

Proofs of the statements of the following theorem
can be found e.g. in \cite[Sect. 3.6]{Engelking1ed}.

\begin{teo}\label{teo:stonecechprop}
Suppose, $X$ is a compact Hausdorff space with
a dense subset $Y$.
Then the following properties are equivalent:
\begin{enumerate}[\rm 1)]
    \item $X$ is the Stone-\v Cech compactification $\b Y$;
    \item any bounded continuous
    function on $Y$ can be extended to a continuous function on
    $X$.
 \end{enumerate}
\end{teo}

\section{Hilbert $C^*$-modules preliminaries}\label{sec:Hilbmodprelim}

Recall that the standard Hilbert $C^*$-module $H_A=l_2(A)$ is the
set of all sequences $(a_1,a_2,\ldots)$, $a_1,a_2,\ldots\in A$,
such that the series $\sum_{i=1}^\infty a_i^*a_i$ is norm
convergent in $A$.

For $H_A=l_2(C(X))$ the dual module can be described as follows
(see, e.g. \cite{MaTroBook}, Prop. 2.5.5):
\begin{equation}\label{eq:formofdual}
H'_A=\biggl\{f=(f_1,f_2,\dots),\: f_i\in C(X), \: \sup_N
\Bigl\|\sum_{i=1}^N f^*_i f_i\Bigr\|<\infty \biggr\},\:
\|f\|^2=\sup_{N} \Bigl\|\sum_{i=1}^N f^*_i f_i\Bigr\|.
\end{equation}
(By $\|\cdot\|$ we denote the sup norm on $X$.)

Unfortunately, there is no similar description of the second dual
module $H_A''$ for general $A$, but in the commutative case we
have some results on elements of $H_A''$. The proof of the
following statement is close to an argument of
\cite{TrofimovUMN1987}.

\begin{lem}\label{lem:extracttrofimov}
Suppose $(F_1,F_2,\dots)\in H'_A$ represents an element $F\in
H''_A$. Let $E$ be the continuity set of the point-wise limit
$\F(x):=\sum_i F^*_i(x)F_i(x)$. Then, for any $x_0\in E$ and any
$f\in H'_A$, the limit of $\sum_i F^*_i(x_0)f_i(x_0)$ equals the
value of the continuous function $F(f)$ at this point.
\end{lem}

\begin{rk}
By Theorem \ref{teo:Bairefunc} the set $E$ is dense in $X$.
\end{rk}

\begin{proof}
Take $x_0 \in E$ and $\e>0$.
Choose a neighborhood $U_0 \ni x_0$
such that
$|\F(x)-\F(x_0)|<\e^2$
for any $x\in U_0$. Choose $N$ such that
$ \sum_{i=N+1}^\infty F^*_i(x_0)F_i(x_0)< \e^2$.
Choose a neighborhood $U_1\ss U_0$ of $x_0$ such that
$$
\Bigl|\sum_{i=1}^N F^*_i(x)F_i(x)- \sum_{i=1}^N
F^*_i(x_0)F_i(x_0)\Bigr|< \e^2\qquad \forall\,x\in U_1
$$
(this is possible because of continuity of this finite sum of
continuous functions). Then, for any $x\in U_1$,
$$
\Bigl|\sum_{i=N+1}^\infty F^*_i(x)F_i(x)\Bigr|
=\Bigl|\F(x)-\sum_{i=1}^N F^*_i(x)F_i(x)\Bigr|
$$
$$\le |\F(x)-\F(x_0)|
+\Bigl|\F(x_0)-\sum_{i=1}^N F^*_i(x_0)F_i(x_0)\Bigr|
+\Bigl|\sum_{i=1}^N F^*_i(x)F_i(x)-\sum_{i=1}^N
F^*_i(x_0)F_i(x_0)\Bigr|
$$
$$
<3\,\e^2.
$$
Because of the isometric embedding $H''_A \ss H'_A$ (cf.
\cite{Paschke2}) this means that, for any continuous function
$\l:X\to [0,1]$ with $\supp \l\ss U_1$ and $\l(x_0)=1$, we have
the following estimate of the norm of an element of $H''_A$:
$$
\Bigl\|\l\,F- \sum_{i=1}^N \l F^*_i \wh e_i \Bigr\|<\sqrt{3}\e,
$$
where $\wh e_i$ are the images of the standard basis elements
under the natural isometric inclusion $H_A\ss H''_A$. For any $f\in H'_A$
$$
\sqrt{3}\e\,\|f\|> \biggl|\Bigl.\Bigl(\l\,F- \sum_{i=1}^N \l F^*_i
\wh e_i\Bigr)(f)\Bigr|_{x_0} \biggr| = \biggl|F(f)(x_0)-
\sum_{i=1}^N  F^*_i(x_0) f_i(x_0)\biggr|.
$$
\end{proof}

\begin{lem}\label{lem:characseconddual}
A sequence $(F_1,F_2,\dots)\in H'_A$ defines an element $F$ of
$H''_A$ if and only if for each $f\in H'_A$, there exists a
continuous function $\a_f$ such that the point-wise limit of the
series $\sum_i F^*_i f_i$ coincides with $\a_f$ on the dense set
$E$ of continuity points of $\sum_i F^*_i F_i$. In this case
$F(f)=\a_f$.
\end{lem}

\begin{rk}
The mentioned point-wise limit always exists because at a point
all our sequences become $l_2(\C)$-sequences.
\end{rk}

\begin{proof}
The ``only if'' statement was proved in
Lemma \ref{lem:extracttrofimov}.

Conversely, let us define a functional $F$ by the formula
$F(f)=\a_f$. Evidently, it is an $A$-functional defined on $H'_A$.
It is bounded by the Cauchy-Buniakovskiy inequality. It remains to
show that, for the natural isometric embedding
$H''_A\hookrightarrow H'_A$, $F\mapsto \til F$, the element $\til
F$ corresponds to $(F_1,F_2,\dots)$, i.e. $\til F(e_i)=F^*_i$.
Indeed, $\til F(e_i)=F(\wh e_i)=\sum_k F^*_k (\wh e_i)_k=F^*_i$,
because in this case the point-wise limit is everywhere
continuous. Here $(\wh e_i)_k$ denotes the $k$-component of $\wh
e_i$.
\end{proof}

\section{A sufficient property for $l_2(C(X))$ to be $C^*$-reflexive}\label{sect:strongtrofimov}
We start with a proof of a stronger version of
\cite{TrofimovUMN1987} (see \cite{FMT2} for a corrected version of
\cite{TrofimovUMN1987}).

\begin{teo}\label{teo:strongtrof}
Suppose a compact Hausdorff space $X$ does not contain a copy of
the Stone-\v Cech compactification $\b\N$ of natural numbers $\N$
as a closed subset. Then $l_2(C(X))$ is $C^*$-reflexive.
\end{teo}

\begin{proof}
Denote for brevity $A:=C(X)$, $H_A:=l_2(C(X))$.
Since $H''_A\ss H'_A$, its elements are represented by
series as in (\ref{eq:formofdual}). Such an element is in $H_A$
if and only if this series is norm-convergent.

Let $F=(F_1,F_2,\dots)\in H''_A$ and set
$$
K_F=\inf_k\sup_{m>k}\sup_{x\in X}\sum_{i=k}^m|F_i(x)|^2
=\inf_k\sup_{m>k}\Bigl\|\sum_{i=k}^m|F_i|^2\Bigr\|.
$$
Obviously, $K_F\leq\|F\|^2$, where $\|F\|^2$ is the least number
$C$ such that $\sup_{x\in X}\sum_{i=1}^\infty|F_i(x)|^2\leq C$. By
the Cauchy criterion, $K_F=0$ if and only if $(F_1,F_2,\dots)\in
H_A$.

We will argue as follows: we will suppose that $H''_A\ne H_A$ and
will prove that $\b\N\ss X$. So we have an element
$(F_1,F_2,\dots)\in H''_A$ such that $K_F>0$.

There exists a number $m(1)$ such that the estimate
$$
\sum_{i=1}^{m(1)-1}|F_i(x)|^2>\|F\|^2-K_F/3
$$
holds for at least one $x\in X$.

Set
$$
U_1=\Bigl\lbrace x\in
X:\sum_{i=1}^{m(1)-1}|F_i(x)|^2>\|F\|^2-K_F/3 \Bigr\rbrace\subset
X.
$$

Set $F^{(1)}=F$,
$F^{(2)}=(0,\ldots,0,F_{m(1)},F_{m(1)+1},\ldots)$, where the first
$m(1)-1$ terms are zeroes. Then $F^{(2)}\in H''_A\setminus H_A$
and $K_{F^{(2)}}=K_F\leq\|F^{(2)}\|^2$.

There exists a number $m(2)>m(1)$ such that the estimate
$$
\sum_{i=m(1)}^{m(2)-1}|F_i(x)|^2>\|F^{(2)}\|^2-K_F/3
$$
holds for at least one $x\in X$.

Set
$$
U_2=\Bigl\lbrace x\in
X:\sum_{i=m(1)}^{m(2)-1}|F_i(x)|^2>\|F^{(2)}\|^2-K_F/3
\Bigr\rbrace\subset X.
$$

Proceeding as above, we get an increasing sequence of numbers
$m(k)$ and a sequence of non-empty open sets $U_k\subset X$ such
that
$$
U_k=\Bigl\lbrace x\in
X:\sum_{i=m(k-1)}^{m(k)-1}|F_i(x)|^2>\|F^{(k)}\|^2-K_F/3\Bigr\rbrace.
$$

Suppose that $\overline{U}_j\cap \overline{U}_l\neq\emptyset$ for
some $j,l$, $j<l$. Take $x_0\in \overline{U}_j\cap
\overline{U}_l$. Then
\begin{equation}\label{la1}
\sum_{i=m(j-1)}^{m(j)-1}|F_i(x_0)|^2\geq \|F^{(j)}\|^2-K_F/3;
\end{equation}
\begin{equation}\label{la2}
\sum_{i=m(l-1)}^{m(l)-1}|F_i(x_0)|^2\geq \|F^{(l)}\|^2-K_F/3\geq
K_F-K_F/3=2K_F/3.
\end{equation}

Summing up (\ref{la1}) and (\ref{la2}), we get
$$
\|F^{(j)}\|^2\geq \sum_{i=m(j-1)}^{m(l)-1}|F_i(x_0)|^2\geq
\|F^{(j)}\|^2-K_F/3+2K_F/3=\|F^{(j)}\|^2+K_F/3.
$$

The obtained contradiction proves that the open sets $U_k$,
$k\in\mathbb N$, (and their closures) do not intersect. Choose a
sequence of points $x_k\in U_k$.

If $E\subset X$ is the (dense) set of continuity points of $\sum_i
F^*_i F_i$ then one can assume also that $x_k\in E$ for each $k$.

Let $\N:=\{x_1,x_2,\dots\}$. We wish to show that the closure
$\ov\N$ of $\N$ in $X$ is homeomorphic to $\b\N$. This is equivalent
(see Theorem \ref{teo:stonecechprop})
to the following property: any bounded function on $\N$ can
be extended to a continuous function on $\ov\N$.

Our functional $F$ should be able to be evaluated on elements of $H'_A$.
In particular, take any bounded sequence $\{\l_k\}$, $\l_k\in \C$.
Choose functions $g_k:X \to [0,1]$, $\supp g_k\ss U_k$, $g_k(x_k)=1$.
Then the sequence
\begin{eqnarray*}
(f_1,f_2,\dots)&=&(\l_1 \|F_1(x_1)\|g_1,\dots, \l_1 \|F_{m(1)}(x_1)\|g_1,\\
&&\quad\l_2 \|F_{m(1)+1}(x_2)\|g_2,\dots, \l_2 \|F_{m(2)}(x_2)\|g_2,\dots)
\end{eqnarray*}
belongs to $H'_A$. By Lemma \ref{lem:extracttrofimov}, the series
$\sum_i F^*_i f_i$ converges over $\N$ point-wise to a continuous
function $F(f)\in C(X)$. In our case,
$$
F(f)(x_k)=\l_k\cdot \sum_{i=m(k)}^{m(k+1)-1} F^*_i(x_k) F_i(x_k).
$$
Thus, varying the sequence $\{\l_k\}$, we can obtain any bounded
sequence of complex numbers, as the sequence of values
$F(f)(x_k)$. Therefore, any bounded function on $\N$, which is
automatically continuous on $\N$, should be extendable to a
contuinuous function on $\ov\N\subset X$ and on entire $X$.
\end{proof}

\section{A criterion for $C^*$-reflexivity}\label{sec:criter}
A more careful analysis of the argument in the previous theorem
implies that instead of embedding $\N$ (and then $\b\N$), we
should embed something coarsely equivalent to $\N$, but more
compatible with the topology on $X$.

Recall that if $\{A_k\}$ is a sequence of Banach spaces then one
can form their direct product $\prod_k A_k$ (resp. direct sum
$\oplus_k A_k$), which is the Banach space of all bounded sequences
$(a_1,a_2,\ldots)$, $a_k\in A_k$, (resp. of all sequences with
$\lim_{k\to\infty}\|a_k\|=0$) with the norm
$\|(a_1,a_2,\ldots)\|=\sup_k\|a_k\|$. If all $A_k$ are
$C^*$-algebras then both $\oplus_k A_k$ and $\prod_k A_k$ are
$C^*$-algebras.

Let $U\subset X$ be an open subset. Then there is a canonical
inclusion of $C_0(U)=\Ker\bigl(C(X)\to C(X\setminus U)\bigr)$ into
$C(X)$. For a sequence $\{U_k\}$, $U_k\subset X$, $k\in\N$, of
open disjoint sets, there is always a canonical inclusion
$\oplus_k C_0(U_k)\subset C(X)$. Sometimes this canonical
inclusion can be extended to an inclusion of $\prod_kC_0(U_k)$
into $C(X)$. In this case we call such inclusion canonical as
well.

Existence of such inclusion of ideals can be expressed in
topological terms: the canonical inclusion of $\sqcup_kU_k$ in $X$
extends to the canonical inclusion of the Gelfand spectrum $Y$ of
$\prod_kC_0(U_k)$ into $X$.

\begin{ex}
Let $X=[0,1]$, $U_k=(\frac{1}{2^{k+1}},\frac{1}{2^k})$. Then the
inclusion $\oplus_k C_0(U_k)\subset C(X)$ doesn't extend to an
inclusion of $\prod_kC_0(U_k)$. Indeed, if we take $f_k\in
C_0(U_k)$ with $\|f_k\|=1$ then the function on $X$ that coincides
with $f_k$ on each $U_k$ is not continuous on X.

\end{ex}

\begin{ex}
Let $X=\beta\N$, $U_k=\{k\}\in\N$. Then $C_0(U_k)=\mathbb C$, and
there is a canonical inclusion of $\prod_kC_0(U_k)$ into
$l^\infty=C(\beta\N)$ (in fact, they coincide).

\end{ex}

\begin{lem}
Let $C(X)$ be not $C^*$-reflexive. Then there exists a sequence
$\{U_k\}$ of disjoint open subsets of $X$ such that $\prod_k
C_0(U_k)$ is canonically included into $C(X)$.
\end{lem}

\begin{proof}
Let $(F_1,F_2,\ldots)\in H''_A\setminus H_A$. As in the proof of
Theorem \ref{teo:strongtrof}, one can construct
\begin{itemize}
\item
a number $K>0$ (one can take $K=2K_F/3$);
\item
an increasing sequence $\{m(k)\}_{k\in\N}$ of integers;
\item
a sequence $U_1,U_2,\ldots$ of open subsets of $X$
\end{itemize}
such that
\begin{enumerate}
\item
$\overline{U}_i\cap\overline{U}_j=\emptyset$ if $i\neq j$;
\item
$K<\sum_{i-m(k)}^{m(k+1)-1} F_i^*(x)F_i(x)\leq \norm{F}^2$ for any
$x\in U_k$.

\end{enumerate}



\medskip
Let $\lambda_k\in C_0(U_k)$. Note that $\lambda_k\cdot F\in
C_0(U_k)\subset C(X)$ for any $F\in C(X)$. Note that the function
$g_k(x)=\sum_{i-m(k)}^{m(k+1)-1} F_i^*(x)F_i(x)$ is invertible on
$U_k$. For a sequence $\lambda=(\lambda_1,\lambda_2,\ldots)$, set
$$
f^\lambda=(f_1,f_2,\ldots)=(\lambda_1F_1g_1^{-1},\ldots,\lambda_1
F_{m(1)}g_1^{-1},\lambda_2 F_{m(1)+1}g_2^{-1},\ldots, \lambda_2
F_{m(2)}g_2^{-1},\ldots).
$$
If the sequence $\lambda$ is bounded (i.e. lies in $\prod_k
C_0(U_k)$) then $f^\lambda\in H'_A$.

Let $Y=\overline{\sqcup_kU_k}\setminus \sqcup_k U_k$. Then
$X\setminus Y$ is dense in $X$ and, for any $x\in X\setminus Y$,
the series $\sum_i F_i^*(x)f_i(x)$ converges either to 0, if $x\in
X\setminus \sqcup_k U_k$, or to $\lambda_k(x)$, if $x\in U_k$.

Define a map $\prod_kC_0(U_k)\to C(X)$ by $\lambda\mapsto
F(f^\lambda)$. It is well-defined due to continuity of $F(f)$ for
each $f\in H'_A$. And it is obviously injective and coincides with
the canonical inclusion of each $C_0(U_k)$ into $C(X)$.
\end{proof}

\begin{lem}
Let there exist a sequence $\{I_k\}$ of non-trivial left ideals in
a $C^*$-algebra $A$ such that
\begin{enumerate}[\rm (1)]
\item
$I_k^*I_l=0$ whenever $k\neq l$;
\item
$\prod_k I_k$ canonically embeds into $A$.
\end{enumerate}
Then $A$ is not $C^*$-reflexive.

\end{lem}
\begin{proof}

Take $a_k\in I_k$ such that $\|a_k\|=1$. Let $F=(a_1,a_2.\ldots)$.
As the series $\sum_k a_k^*a_k$ doesn't converge in norm, $F\notin
H_A=l_2(A)$. Let us show that $F\in H''_A$. Take some
$f=(f_1,f_2,\ldots)\in H'_A$. Then we can define $F(f)$ as
$F(f)=\sum_k a_k^*f_k:=(a_1^*f_1,a_2^*f_2,\ldots)$. As $I_k$ is a
left ideal, so $a_k^*f_k\in I_k$. As $f\in H'_A$, so the sequence
$(a_1^*f_,a_2^*f_2,\ldots)$ is bounded, hence lies in $\prod_k
I_k$, hence, by assumption, in $A$. Thus $F(f)\in A$ is
well-defined.

\end{proof}

So, we have proved the following theorem.

\begin{teo}\label{teo:criteriumwithsystem}
The module $l_2(C(X))$ is not $C^*$-reflexive if and only if there
exists a sequence $\{U_k\}$ of open pairwise non-intersecting
non-empty sets in $X$ such that
$$
\prod_k C_0(U_k)\subset C(X).
$$
\end{teo}

\begin{proof}
This follows from the two preceding lemmas. If $A=C(X)$ then
$C_0(U_k)$ are the (left) ideals required in the second Lemma.
\end{proof}

Now, keeping in mind the Kasparov stabilization theorem
and some evident topological argument, we can reformulate
this theorem in the following way.

\begin{teo}\label{teo:criteriumwithset}
Any countable generated Hilbert $C^*$-module over $C(X)$ is
$C^*$-reflexive if and only if there does not exist any sequence
of orthogonal ideals $I_k\in C(X)$ such that $\prod_k I_k\subset
C(X)$.

\end{teo}

We conjecture that the same condition gives a criterion for
$C^*$-reflexivity for general (non-commutative) $C^*$-algebras.

\section{An example}\label{sec:example}

Let $A$ be the $C^*$-subalgebra of $l^\infty$ that consists of all
sequences $\{a_n\}_{n\in\N}$ such that
$\lim_{n\to\infty}|a_{n+1}-a_n|=0$. This $C^*$-algebra is the
algebra of all continuous functions on the {\em Higson
compactification} $\nu\N$ of $\N$ \cite{Roe}.

\begin{teo}
The $C^*$-algebra $A=C(\nu\N)$ is $C^*$-reflexive.

\end{teo}
\begin{proof}
Assume the contrary. Then there exist disjoint open subsets $U_k$,
$k\in\N$, of $\nu\N$ such that $\prod_k C_0(U_k)\subset A$. Being
an open set of $\nu\N$, each $U_k$ contains at least one point of
$\N$. Let $n_k\in U_k$ be such a point. Each point of $\N$ is also
an open set of $\nu\N$, and $\mathbb C\cong C_0(\{n_k\})\subset
C_0(U_k)$. Therefore, $\prod_k C_0(\{n_k\})\subset A$. Take an
arbitrary sequence
$$
\{a_n\}_{n\in\N}\in\prod_kC_0(\{n_k\}).
$$

Set $\Mm=\N\setminus\cup_k\{n_k\}=\{m_1,m_2,\ldots\}$. As
$\{a_n\}_{n\in\N}\in\prod_kC_0(\{n_k\})$, so $a_n=0$ for any
$n\in\Mm$.

If $\Mm$ is finite then the sequence
$\{a_n\}_{n\in\N}\in\prod_kC_0(\{n_k\})$ is (modulo several first
terms) an arbitrary bounded sequence, which contradicts that this
sequence lies in $A$. If $\Mm$ is infinite then for each $n$ there
is a number $m>n$ such that one can find integers $k_1$ and $k_2$
such that $n_{k_1}>m$, $m_{k_2}>m$ and $|n_{k_1}-m_{k_2}|=1$. As
$a_{m_{k_2}}=0$ and $a_{n_{k_1}}$ may take an arbitrary value, so
we get a contradiction with the condition
$\lim_{n\to\infty}|a_{n+1}-a_n|=0$. Getting contradictions in both
cases, we conclude that our assumption was false.
\end{proof}

\begin{rk}
Note that there exists a (non-canonical) inclusion
$\beta\N\subset\nu\N$. Indeed, let $\{n_k\}_{k\in\N}$ be an
increasing sequence of integers such that
$$
\lim_{k\to\infty}n_{k+1}-n_k=\infty.
$$
Set $b_k=a_{n_k}$. The map
$\{a_n\}_{n\in\N}\mapsto\{b_k\}_{k\in\N}$ gives a $*$-homomorphism
from $A$ to $l^\infty$, and an easy check shows surjectivity of
this map. Therefore, the map $k\mapsto n_k$ extends to a
continuous injective map $\beta\N\to\nu\N$.

This shows that our sufficient condition from Section
\ref{sect:strongtrofimov} is not a necessary condition.
\end{rk}


\begin{thebibliography}{10}

\bibitem{Engelking1ed}
R. Engelking, \emph{General topology}, PWN---Polish Scientific
Publishers,
  Warsaw, 1977.

\bibitem{FrankZAA1990}
M.~Frank, \emph{Self-duality and {$C\sp *$}-reflexivity of {H}ilbert {$C\sp
  *$}-moduli}, Z. Anal. Anwendungen \textbf{9} (1990), no.~2, 165--176.


\bibitem{FMTZAA}
M.~Frank, V.~M. Manuilov, and E.~V. Troitsky, \emph{On conditional expectations
  arising from group actions}, Z. Anal. Anwendungen \textbf{16} (1997), no.~4,
  831--850.

\bibitem{FMT2}
\bysame, \emph{Hilbert $C^*$-modules from group actions: beyond
the finite
  orbits case}, arXiv:0903.1741, 2009.

\bibitem{KaspJO}
G.~G. Kasparov, \emph{Hilbert {$C\sp{\ast} $}-modules: theorems of
  {S}tinespring and {V}oiculescu}, J. Operator Theory \textbf{4} (1980), no.~1,
  133--150.

\bibitem{MaTroBook}
V.~M. Manuilov, E.~V. Troitsky, \emph{Hilbert {$C\sp *$}-modules},
  Translations of Mathematical Monographs, vol. 226, American Mathematical
  Society, Providence, RI, 2005.

\bibitem{Mishchenko-Trudy_MIAN}
A.~S. Mishchenko, \emph{Representations of compact groups in {H}ilbert modules
  over {$C\sp{\ast} $}-algebras}, Trudy Mat. Inst. Steklov. \textbf{166}
  (1984), 161--176.

\bibitem{munkres2000}
J.~R. Munkres, \emph{Topology}, 2nd ed., Prentice Hall, 2000.

\bibitem{Paschke2}
W.~L. Paschke, \emph{The double ${B}$-dual of an inner product module over a
  ${C^*}$-algebra ${B}$}, Can. J. Math. \textbf{26} (1974), no.~5, 1272--1280.

\bibitem{Paschke}
W.~L. Paschke, \emph{Inner product modules over {$B\sp{\ast}
$}-algebras},
  Trans. Amer. Math. Soc. \textbf{182} (1973), 443--468.

\bibitem{Roe}
J. Roe, \emph{Coarse cohomology and index theory for complete
Riemannian manifolds}, Memoirs Amer. Math. Soc. {\bf 497} (1993).


\bibitem{TrofimovK}
V.~A. Trofimov, \emph{Reflexivity of {H}ilbert modules over an algebra of
  compact operators with associated identity}, Vestnik Moskov. Univ. Ser. I
  Mat. Mekh. (1986), no.~5, 60--64.

\bibitem{TrofimovUMN1987}
\bysame, \emph{Reflexive and self-dual {H}ilbert modules over some {$C\sp
  *$}-algebras}, Uspekhi Mat. Nauk \textbf{42} (1987), no.~2(254), 247--248.


\end{thebibliography}

\def\cprime{$'$} \def\cprime{$'$}
\providecommand{\bysame}{\leavevmode\hbox to3em{\hrulefill}\thinspace}
\providecommand{\MR}{\relax\ifhmode\unskip\space\fi MR }
\providecommand{\MRhref}[2]{%
  \href{http://www.ams.org/mathscinet-getitem?mr=#1}{#2}
}
\providecommand{\href}[2]{#2}

\end{document}